\documentclass[a4paper,10pt]{amsart}
\usepackage{pb-diagram}
\usepackage{xypic}
\usepackage{amsfonts,amscd,amssymb,amsmath,latexsym,amscd,bbm,amsthm,mathrsfs,hyperref}
\usepackage{amscd}
\usepackage{bm}
\usepackage{amsmath}
\usepackage{latexsym}
\usepackage{amsfonts}
\usepackage{amsthm}
\usepackage{indentfirst}
 \newtheorem{thm}{Theorem}[section]
 \newtheorem{cor}[thm]{Corollary}
 \newtheorem{lem}[thm]{Lemma}
 \newtheorem{prop}[thm]{Proposition}
 \newtheorem{defn}[thm]{Definition}
 \newtheorem{ex}[thm]{Example}
 \newtheorem{rem}[thm]{Remark}

 \newcommand{\Hom}{\mathrm{Hom}}

\title{Local cohomology for Gorenstein homologically smooth DG algebras}
\author{X.-F. Mao}
\address{Department of Mathematics, Shanghai University, Shanghai 200444, China}
\email{xuefengmao@shu.edu.cn}

\author{H. Wang}
\address{Department of Mathematics, Shanghai University, Shanghai 200444, China}
\email{happywang97@shu.edu.cn}

\date{}

\begin{document}

\maketitle \def\abstactname{abstact}
\begin{abstract}
In this paper, we introduce the theory of local cohomology and local duality to Notherian connected cochain DG algebras.
 We show that the notion of local
cohomology functor can be used to detect the Gorensteinness of a homologically smooth DG algebra. For any Gorenstein homologically smooth locally finite DG algebra $\mathcal{A}$, we define a group homomorphism $\mathrm{Hdet}: \mathrm{Aut}_{dg}(\mathcal{A})\to k^{\times},$ called the homological determinant.
 As applications, we present a sufficient condition for the invariant DG subalgebra $\mathcal{A}^G$ to be Gorensten, where $\mathcal{A}$ is a homologically smooth DG algebra such that $H(\mathcal{A})$ is a Noetherian AS-Gorenstein graded algebra and $G$ is a finite subgroup of $\mathrm{Aut}_{dg}(\mathcal{A})$. Especially, we can apply this result to  DG down-up algebras and non-trivial DG free algebras generated in two degree-one elements.

\end{abstract}

\maketitle
\section*{introduction}
 Local cohomology is an
indispensable tool in almost all branches of analytic and algebraic
geometry as well as in commutative and combinatorial algebra. Since
the flourishing of non-commutative algebraic geometry in
1990s, people have done a lot to develop a theory of local cohomology in the
context of non-commutative cases.  In \cite{Yek}, Yekutieli
introduced balanced dualizing complexes for non-commutative graded
algebras by using local cohomology functors.  After his work,
J$\o$rgensen \cite{Jor1} generalized the theory of local cohomology
and local duality theorem to non-commutative positively graded Noetherian
algebras which are quotients of Artin-Schelter Gorenstein algebras.
Later, Van den Bergh  gave a more general local duality
formula in \cite{VDB1}. In \cite{JoZ}, local cohomology functor was used to define homological determinants for AS-Gorenstein algebras.
In the invariant theory of Hopf algebra action (and formerly group actions) on non-commutative graded algebras, the notion of homological determinant usually serves as an effective tool in precise analysis (cf. \cite{CWZ, JoZ,JZ,KKZ1, LMZ1}).

In this paper, we are concerned with connected cochain DG $k$-algebras.
 Any connected graded algebra can
be seen as a connected cochain DG algebra with zero differential. The motivation of this paper is to
introduce an analogous local cohomology theory under the background
of DG homological algebra. Let $\mathcal{A}$ be a connected cochain DG algebra with maximal DG ideal $\frak{m}=\mathcal{A}^{\ge 1}$. For any DG $\mathcal{A}$-module $M$, its subset $$\Gamma_{\frak{m}}(M) = \{m\in M\mid \mathcal{A}^{\ge
n}m = 0 ,\,\, \text{for some}\,\, n \ge 1\}$$ is a DG
$\mathcal{A}$-submodule of $M$.  We actually define a left exact covariant $\mathcal{A}$-linear functor $$\Gamma_{\frak{m}}(-)=\varinjlim_{n\ge 1}\Hom_{\mathcal{A}}(\mathcal{A}/\mathcal{A}^{\ge
n},-)$$  from the category $\mathscr{C}(\mathcal{A})$ of left DG $\mathcal{A}$-modules to itself.
It is natural to define its right derived functor $R\Gamma_{\frak{m}}(-)$ as the local
cohomology functor. We emphasize that our definitions of torsion DG modules and the local cohomology functor are quite different from the definitions in \cite{Jor2}.  Due to the differences between DG homological
functors and hyper-homological functors in computation, it seems
difficult to get a complete generalization of the original theorem
of local duality as in \cite[Theorem 2.3]{Jor1}.
 For example, we don't know whether the torsion
functor preserves the homotopically injective property of DG
modules. However, under some mild hypothesis, we prove the following
local duality theorem (see Theorem \ref{localdual}) in DG context.
\\
\begin{bfseries}
Theorem \ A.
\end{bfseries}
Let $M$ be a compact DG module over a locally finite connected cochain DG algebra $\mathcal{A}$.
Then $R\Gamma_{\frak{m}}(M)' \cong
R\Hom_{\mathcal{A}}(M,R\Gamma_{\frak{m}}(\mathcal{A})')$ in $\mathscr{D}(\mathcal{A}\!^{op})$.

Note that we add compact condition on the DG module $M$, and we define the Matlis dual of a graded $k$-vector space $V$ as $$V'=\Hom_k(V,k)=\bigoplus\limits_{i\in \Bbb{Z}}\Hom_{k}(V^i,k).$$
Using the
duality formula of the theorem above, we give characterizations of
Gorensteinness for Noetherian homologically
smooth connected cochain DG algebras. This is the following theorem
(see
Theorem \ref{smoothgoren}).  \\
\begin{bfseries}
Theorem \ B.
\end{bfseries}
Let $\mathcal{A}$ be a Noetherian homologically smooth connected cochain DG
algebra. Then $\mathcal{A}$ is Gorenstein if and only if
$R\Gamma_{\frak{m}}(\mathcal{A})\cong \mathcal{A}'\otimes_{\mathcal{A}} L$ in $\mathscr{D}(\mathcal{A}^e)$,
where $L$ is an invertible DG bi-module over $\mathcal{A}$.

 Applying the characterization of Gorensteinness for Noetherian
homologically smooth connected cochain DG algebras in the theorem
above, we give a proof of the following interesting theorem (see
Theorem \ref{biduality}).
 \\
\begin{bfseries}
Theorem \ C.
\end{bfseries}
Assume that $\mathcal{A}$ is a Noetherian homologically smooth and Gorenstein
 connected cochain DG algebra. Then for any $M\in \mathscr{D}^c(\mathcal{A})$,
$R\Gamma_{\frak{m}}(M)'$ is a compact DG $\mathcal{A}\!^{op}$-module. And
there is a pair of contravariant equivalences of categories,
\begin{align*}
\xymatrix{&\mathscr{D}^{c}(\mathcal{A})\quad\quad\ar@<1ex>[r]^{R\Gamma_{\frak{m}(-)'}}&\quad\quad
\mathscr{D}^{c}(\mathcal{A}\!^{op})\ar@<1ex>[l]^{R\Gamma_{\frak{m}^{op}}(-)'}}.
 \end{align*}

 People have done a lot of work
to extend the rich theory of commutative Gorenstein rings to DG
algebras (cf. \cite{FHT1,AF2,DGI,FJ,Dual, MW2}). In this paper, we study the Gorensteinness of invariant DG subalgebras of some nice Gorenstein DG algebras.
 In the literature, there have been many papers on the behavior of the Gorenstein property under the passage to the subring of invariants (cf. \cite{Wat1,Wat2,Hin, JZ, KKZ1}). However, it is infeasible to mimic the proof there since the techniques and tools used in commutative and non-commutative graded setting are not available in DG context.
 For this,
 we need to develop some tools. Let $\mathcal{A}$ be a Gorenstein, homologically smooth and locally finite connected cochain DG algebra. We define homological determinant for any DG algebra automorphism of $\mathcal{A}$ in Definition \ref{DGhdet}. Suppose $G$ is a finite subgroup of the automorphism group of a connected cochain DG algebra $\mathcal{A}$. The invariant DG subalgebra is defined to be
$$\mathcal{A}^G=\{x\in \mathcal{A}\mid g(x)=x, \,\, \text{for all}\,\, g\in G\}.$$
Then $H(G):\{H(g)| g\in G\}$ is a finite subgroup of $\mathrm{Aut}_{gr}H(\mathcal{A})$ and the graded fixed subalgebra $H(\mathcal{A})^{H(G)}$ of $H(\mathcal{A})$ is isomorphic to $H(\mathcal{A}^{G})$ (see Proposition \ref{equal}).
Moreover, we prove the following theroem (see Theorem \ref{fixedsub}).\\
 \begin{bfseries}
Theorem \ D.
\end{bfseries}
Let $\mathcal{A}$ be a homologically smooth locally finite connected cochain DG algebra such that $H(\mathcal{A})$ is a Noetherian AS-Gorenstein graded algebra.
If $G$ is a finite subgroup of $\mathrm{Aut}_{dg}(\mathcal{A})$ such that the homological determinant of each $\sigma \in G$ is $1$,  then the fixed DG subalgebra $\mathcal{A}^G$ is a Gorenstein DG algebra.

In addition, we find that the result in Theorem D can be successfully applied to DG down-up algebras and some
non-trivial DG free algebras (see  Corollary \ref{dgducase} and Corollary \ref{dgfreecase}).

To close the introduction, let us run through some notation and conventions. Let $W$ be a graded $k$-vector space.  For any  $j\in\Bbb{Z}$, the $j$-th suspension $\Sigma^j W$ of $W$ is a graded vector space defined by $(\Sigma^j W)^i=W^{i+j}$.
We assume that the reader is familiar with basics on DG algebras and DG modules. Most of the notations and conventions on them we take in this paper are same as in \cite{MW1,MW2}.
Let $\mathcal{A}$ be a connected cochain DG algebra. We take the following notations:
\begin{enumerate}
\item $\mathcal{A}^{\#}$ is the underlying graded algebra of $\mathcal{A}$;
\item $\mathrm{Aut}_{dg}(\mathcal{A})$ is the group of all DG algebra automorphisms of $\mathcal{A}$;
\item $\mathcal{A}\!^{op}$ is the opposite DG
algebra of $\mathcal{A}$;
\item $\mathcal{A}^e$ is the enveloping DG algebra of $\mathcal{A}$;
\item $\frak{m}_{\mathcal{A}}$ is the maximal DG ideal $\mathcal{A}^{>0}$ of $\mathcal{A}$;
\item $H(\mathcal{A})$ is the cohomology graded algebra of $\mathcal{A}$;
\item $\mathscr{C}(\mathcal{A})$ is the category of left DG modules over $\mathcal{A}$;
\item $\mathscr{D}(\mathcal{A})$ is the derived category of $\mathscr{C}(\mathcal{A})$;
\item $\mathscr{D^c}(\mathcal{A})$ is the full triangulated subcategory of $\mathscr{D}(\mathcal{A})$ consisting of compact DG $\mathcal{A}$-modules.
\end{enumerate}
Let $z$ be a graded cocycle element in $\mathcal{A}$, we write $\lceil z \rceil$ for the cohomology class in $H(\mathcal{A})$ represented by $z$.
For any $M, N\in \mathscr{C}(\mathcal{A})$ and $\sigma\in \mathrm{Aut}_{dg}(\mathcal{A})$, We denote ${}^{\sigma}M$ by an object in $\mathscr{C}(\mathcal{A})$ such that
${}^{\sigma}M$ has the same underlying complex of $k$-vector spaces as $M$ and its multiplication is defined by $a\cdot m=\sigma(a)m$ for any $a\in \mathcal{A}$ and $m\in M$. A $k$-linear chain map $f:M\to N$ is called a $\sigma$-linear chain map if it gives a morphism of DG $\mathcal{A}$-modules, $f:M\to {}^{\sigma}N$, i.e. if $f(am)=\sigma(a)f(m)$.

Finally, we recall some properties of DG algebras, which will be mentioned in this paper.
 Let $\mathcal{A}$ be a connected cochain DG algebra.
\begin{enumerate}
\item  If $\mathcal{A}^{\#}$ is a Noetherian graded algebra, then $\mathcal{A}$ is called Noetherian;
\item  If $\dim_k \mathcal{A}^i<\infty, \forall i\in \Bbb{Z}
$, then $\mathcal{A}$
is called locally finite;
\item  If $\dim_{k}H(R\Hom_{\mathcal{A}}(k,\mathcal{A}))=1$, then $\mathcal{A}$ is called left Gorenstein (cf. \cite{FHT1}); Similarily, $\mathcal{A}$ is called right Gorenstein, if $\dim_{k}H(R\Hom_{\mathcal{A}\!^{op}}(k,\mathcal{A}))=1$;
\item  If ${}_{\mathcal{A}}k$, or equivalently ${}_{\mathcal{A}^e}\mathcal{A}$, admits a minimal semi-free resolution with a semi-basis concentrated in degree $0$, then $\mathcal{A}$ is called Koszul (cf. \cite{HW});
\item If ${}_{\mathcal{A}}k$, or equivalently the DG $\mathcal{A}^e$-module $\mathcal{A}$ is compact, then we say that $\mathcal{A}$ is a homologically smooth DG algebra(cf. \cite[Corollary 2.7]{MW3}).
\end{enumerate}
Note that each Noetherian connected cochain DG algebra is locally finite. While the converse is generally not true.
In general, a left Gorenstein connected cochain DG algebra  may be not right Gorenstein. However, a homologically smooth DG algebra is left Gorenstein if and only if it is right Gorenstein (see \cite[Remark 6.6]{MW2}). Beside this, we have the following remark for a homologically smooth DG algebra (see \cite{MW2}).
\begin{rem}\label{HSM}
Let $\mathcal{A}$ be a homologically smooth connected cochain DG algebra and $M\in \mathscr{D}(\mathcal{A})$. Then
\begin{enumerate}
\item $R\Hom_{\mathcal{A}}(k,\mathcal{A})\stackrel{\simeq}{\to}\Sigma^i k_{\mathcal{A}}$ if and only if
$R\Hom_{\mathcal{A}\!^{op}}(k,\mathcal{A})\stackrel{\simeq}{\to}\Sigma^i{}_{\mathcal{A}}k$;
\item $M\in \mathscr{D}^c(\mathcal{A})$ if $\dim_kH(M)<\infty$;
\item $H(\mathcal{A})$ is not bounded above if $\mathcal{A}$ is not quasi-isomorphic to $k$.
\end{enumerate}

\end{rem}
\section{some basics on dg modules and dg algebras}
In this section, we give some lemmas on DG modules and DG algebras, which will be used in the later sections.
\begin{lem}\label{semi-inj}
Let $I$ be a semi-injective DG $\mathcal{A}^e$-module. Then $I$ is
semi-injective both as a DG $\mathcal{A}$-module and as a DG $\mathcal{A}^{op}$-module.
\end{lem}

\begin{proof}
 Since
$\Hom_{\mathcal{A}}(-,I)\cong \Hom_{\mathcal{A}}(-,\Hom_{\mathcal{A}^e}(\mathcal{A}^e,I))\cong
\Hom_{\mathcal{A}^e}(\mathcal{A}\otimes-, I),$ it is easy to check that  $I$ is a
semi-injective DG $\mathcal{A}$-module. Similarly, we can prove that $I$ is
also a semi-injective DG $\mathcal{A}^{op}$-module.
\end{proof}

\begin{lem}\label{tensorcong}
Let $\mathcal{A}$ and $\mathcal{B}$ be two connected cochain DG algebras. Then for any DG
$\mathcal{A}$-module $P$, DG $\mathcal{B}$-module $Q$ and DG $\mathcal{A}\otimes \mathcal{B}$-module $N$,
the chain map \begin{align*} \phi: \Hom_{\mathcal{A}\otimes \mathcal{B}}(P\otimes Q, N)&
\to \Hom_{\mathcal{A}}(P, \Hom_{\mathcal{B}}(Q,N))\\
& f\mapsto \phi(f): x\to f_x: y\to f(x\otimes y)\\
\end{align*}
is an isomorphism, whose inverse is
\begin{align*}
\psi:\Hom_{\mathcal{A}}(P, \Hom_{\mathcal{B}}(Q,N))&\to  \Hom_{\mathcal{A}\otimes \mathcal{B}}(P\otimes Q, N)\\
                                                  g &\mapsto \psi(g): x\otimes y \mapsto g(x)(y).
\end{align*}
\end{lem}

\begin{proof}

For any $f\in  \Hom_{\mathcal{A}\otimes \mathcal{B}}(P\otimes Q, N)$ and $x\in P$,
we should prove the following statements
\begin{enumerate}
\item $f_x\in  \Hom_{\mathcal{B}}(Q,N)$; \\
\item $\phi(f)\in \Hom_{\mathcal{A}}(P,\Hom_{\mathcal{B}}(Q,N))$;\\
\item $\phi$ is a chain map.
\end{enumerate}
For any $a,b\in B$, $x\in P$ and $y\in Q$, we have  $f_x\in  \Hom_{\mathcal{B}}(Q,N)$ since
\begin{align*}
f_x(by)&=f(x\otimes by)=f[(1\otimes b)(x\otimes y)]\\
       &=(1\otimes b)f(x\otimes y)=(1\otimes b)f_x(y)=bf_x(y).
\end{align*}
Similarly, $\phi(f)(ax)=a\phi(f)(x)$ since
\begin{align*}
&\phi(f)(ax)=f_{ax}:y\mapsto f(ax\otimes y)=(a\otimes 1)f(x\otimes y)\\
&a\phi(f)(x)=af_x: y\mapsto af(x\otimes y)=(a\otimes 1)f(x\otimes y).
\end{align*}
Furthermore,
\begin{align*}
&[\partial_{\Hom}\circ \phi(f)](x)(y)=[\partial_{\Hom_{\mathcal{B}}(Q,N)}\circ f_x-(-1)^{|f|}\phi(f)(\partial_P(x))](y)\\
&=\partial_N\circ f_x(y)-(-1)^{|f_x|}f_x[\partial_Q(y)]-(-1)^{|f|}f(\partial_P(x)\otimes y)\\
&=\partial_N(f(x\otimes y))-(-1)^{|f|+|x|}f[x\otimes \partial_Q(y)]-(-1)^{|f|}f[\partial_P(x)\otimes y]
\end{align*}
and
\begin{align*}
&[\phi\circ \partial_{\Hom}(f)](x)(y)=[\partial_{\Hom}(f)]_x(y)=[\partial_N\circ f-(-1)^{|f|}f\circ\partial_{P\otimes Q}]_x(y)\\
&=\partial_N(f(x\otimes y))-(-1)^{|f|}f[\partial_P(x)\otimes y+(-1)^{|x|}x\otimes \partial_Q(y)]\\
&=\partial_N(f(x\otimes y))-(-1)^{|f|}f[\partial_P(x)\otimes y]-(-1)^{|f|+|x|}f[x\otimes \partial_Q(y)].
\end{align*}
So $\phi$ is a chain map. Similarly, we can show that $\psi$ is chain map. Furthermore,
\begin{align*}
&\psi\circ \phi (f) (x\otimes y)= \phi(f)(x)(y)=f_x(y)=f(x\otimes y),\\
&\phi\circ \psi (g) (x)(y)= \psi (g)_x(y)= \psi (g)(x\otimes y)=g(x)(y),
\end{align*}
for any $f\in \Hom_{\mathcal{A}}(P,\Hom_{\mathcal{B}}(Q,N))$ and $g\in \Hom_{\mathcal{A}}(P, \Hom_{\mathcal{B}}(Q,N))$, $x\in P$ and $y\in Q$.
Thus $\phi\circ \psi=\mathrm{id}$ and $\psi\circ \phi=\mathrm{id}$.
\end{proof}

\begin{lem}\label{tenprop}
Let $\mathcal{A}$ and $\mathcal{B}$ be two homologically smooth connected cochain DG
algebras. Then
\begin{enumerate}
\item  $\mathcal{A}\otimes \mathcal{B}$ is a homologically smooth
connected cochain DG algebra;
\item  $\mathcal{A}\otimes \mathcal{B}$ is Gornstein if both $\mathcal{A}$ and $\mathcal{B}$ are Gorenstein DG algebras.
 \end{enumerate}
\end{lem}

\begin{proof}
(1)
Let $P$ and $Q$ be the minimal semi-free resolutions of ${}_{\mathcal{A}}k$ and
${}_{\mathcal{B}}k$ respectively. Then $P\otimes Q$ is a minimal semi-free
resolution of ${}_{\mathcal{A}\otimes \mathcal{B}}k$. By Lemma \ref{tensorcong}, we have
$$H(\Hom_{\mathcal{A}\otimes \mathcal{B}}(P\otimes Q, k)) \cong H(\Hom_{\mathcal{A}}(P,
\Hom_{\mathcal{B}}(Q,k))).$$ Since ${}_{\mathcal{A}}k$ and ${}_{\mathcal{B}}k$ are compact by the
assumption, both $P$ and $Q$ have a finite semi-basis. This implies
that $\dim_kH(\Hom_{\mathcal{A}\otimes \mathcal{B}}(P\otimes Q,k))<\infty.$ Hence
$\mathcal{A}\otimes \mathcal{B}$ is also a homologically smooth connected cochain DG
algebra.

(2)By (1), $\mathcal{A}\otimes \mathcal{B}$ is homologically
smooth. Let $P$ and $Q$ be the minimal
semi-free resolutions of ${}_{\mathcal{A}}k$ and ${}_{\mathcal{B}}k$ respectively. Then
$P\otimes Q$ is a minimal semi-free resolution of ${}_{\mathcal{A}\otimes
\mathcal{B}}k$. We have
\begin{align*}
H(\Hom_{\mathcal{A}\otimes
\mathcal{B}}(P\otimes Q, \mathcal{A}\otimes \mathcal{B})) & \stackrel{(a)}{\cong} H(\Hom_{\mathcal{A}}(P, \Hom_{\mathcal{B}}(Q,\mathcal{A}\otimes
\mathcal{B})))\\
&\stackrel{(b)}{\cong} H(\Hom_{\mathcal{A}}(P, \mathcal{A}\otimes \Hom_{\mathcal{B}}(Q,\mathcal{B})) \\
&\stackrel{(c)}{\cong} H(\Hom_{\mathcal{A}}(P,\mathcal{A})\otimes \Hom_{\mathcal{B}}(Q,\mathcal{B}))\\
&\stackrel{(d)}{\cong} H(\Hom_{\mathcal{A}}(P,\mathcal{A}))\otimes H(\Hom_{\mathcal{B}}(Q,\mathcal{B})),
\end{align*}
where $(a)$ is by  Lemma \ref{tensorcong}, $(b)$ and $(c)$ are obtained by the fact that both $P$ and $Q$ have a finite semi-basis, and $(d)$ is by the well known K\"{u}nneth formula.
On the other hand, $\dim_k H(\Hom_{\mathcal{A}}(P,\mathcal{A})) = \dim_k H(\Hom_{\mathcal{B}}(Q,\mathcal{B}))=
1$ since $\mathcal{A}$ and $\mathcal{B}$ are Gorenstein. Thus $\dim_kH(\Hom_{\mathcal{A}\otimes
\mathcal{B}}(P\otimes Q, \mathcal{A}\otimes \mathcal{B}))=1$. This implies that $\mathcal{A}\otimes \mathcal{B}$ is also Gorenstein.
\end{proof}

\begin{defn}\label{dginvert}
{\rm A DG $\mathcal{A}^e$-module $L$ is called an invertible DG $\mathcal{A}$-bimodule if
there is some DG $\mathcal{A}^e$-module $L^{\vee}$ such that $L\otimes_{\mathcal{A}}
L^{\vee}\cong L^{\vee}\otimes_{\mathcal{A}} L \cong \mathcal{A}$  over $\mathcal{A}^e$.}
\end{defn}

For any $\phi\in \mathrm{Aut}_{dg}(\mathcal{A})$ and $i\in \Bbb{Z}$, we difine an invertible DG
bimodule $\Sigma^i \mathcal{A}(\phi)$ as follows. As a left DG $\mathcal{A}$-module,
$\Sigma^i\mathcal{A}(\phi)$ is the DG $\mathcal{A}$-module $\Sigma^i\mathcal{A}$ with a
generator $e=1$ in degree $-i$. The right multiplication is
given by the rule $ea =(-1)^{i\cdot|a|}
 \phi(a)e, a\in \mathcal{A}$.
 One sees that
 $\Sigma^i\mathcal{A}(\phi)\otimes_{\mathcal{A}}
 \Sigma^u \mathcal{A}(\psi) \cong \Sigma^{i+u}\mathcal{A}(\phi\circ \psi)$, for any $\psi \in \mathrm{Aut}_{dg}(\mathcal{A})$.
\begin{lem}\label{hom}
Let $L\cong \Sigma^i\mathcal{A}(\phi)$ be an invertible DG bimodule over a connected cochain DG algebra $\mathcal{A}$, for some
$i\in \Bbb{Z}$ and $
\phi\in \mathrm{Aut}_{dg}(\mathcal{A})$.
 Then $\Hom_{\mathcal{A}}(L,L)\cong \mathcal{A}$
as a DG $\mathcal{A}^e$-module.
\end{lem}
\begin{proof}
We have the following isomorphisms of DG $\mathcal{A}^e$-modules:
$$\Hom_{\mathcal{A}}(L,L)\cong \Hom_{\mathcal{A}}(L,\mathcal{A})\otimes_{\mathcal{A}}L \cong
\mathcal{A}(\phi^{-1})\otimes_{\mathcal{A}} \mathcal{A}(\phi) \cong \mathcal{A}. $$
\end{proof}

By \cite[Proposition 1.10]{Yek}, any invertible graded bimodule over a connected graded algebra $R$ can be denoted by $\Sigma^iR(\phi)$, for some $i\in \Bbb{Z}$ and $\phi\in \mathrm{Aut}_{gr}(R)$. We attempt to obtain a similar result in DG context. However, the following counter example
indicates that
not all invertible DG $\mathcal{A}$-bimodules have the form $\Sigma^i \mathcal{A}(\psi)$, $i\in \Bbb{Z}$ and $\psi\in \mathrm{Aut}_{dg}(\mathcal{A})$.
\begin{ex}
Let $\mathcal{A}=k[x]/(x^2)$ with $\partial_{\mathcal{A}}=0$ and $|x|=1$. Assume that  $L$ and $L^{\vee}$ are DG $\mathcal{A}$-bimodules such that
\begin{align*}
&L^{\#}=\mathcal{A}^{\#}e \,\,\text{with}\,\, e\cdot \bar{x}=\bar{x}e, \partial_{L}(e)=\bar{x}e\\
&L^{\vee\#}=\mathcal{A}^{\#}f \,\,\text{with}\,\, f\cdot \bar{x}=\bar{x}f, \partial_{L^{\vee}}(f)=-\bar{x}f.
\end{align*}
Then $\partial_{\otimes}(e\otimes f)=\bar{x}e\otimes f-e\otimes \bar{x}f=e\cdot\bar{x}\otimes f-e\otimes \bar{x}f=0$ in $L\otimes_{\mathcal{A}}L^{\vee}$. One see that
the morphisms
\begin{align*}
\theta: & L\otimes_{\mathcal{A}}L^{\vee}\to \mathcal{A}\\
        &e\otimes f\mapsto  1_{\mathcal{A}}\\
        \end{align*}
and
\begin{align*}
\omega: & L^{\vee}\otimes_{\mathcal{A}}L\to \mathcal{A}\\
        &f\otimes e\mapsto  1_{\mathcal{A}}\\
        \end{align*}
 are both isomorphisms of DG $\mathcal{A}^e$-modules. Therefore, $L$ and $L^{\vee}$ are invertible DG $\mathcal{A}$-bimodules.
 However, $L$ is not isomorphic to $\mathcal{A}(\phi)$ for any $\phi\in \mathrm{Aut}_{dg}(\mathcal{A})$, since $\partial_L(e)=\bar{x}e\neq 0$.
\end{ex}

\begin{lem}\label{onedim}
Let $M$ a DG $\mathcal{A}$-module such that $\dim_kH(M)=1$. Then
$M\cong\Sigma^i {}_{\mathcal{A}}k$ in $\mathscr{D}(\mathcal{A})$, for some integers
$i\in \Bbb{Z}$.
\end{lem}
\begin{proof}
Since $\dim_kH(M)=1$, we have $H(M)=\Sigma^i k$ for some integers
$i\in \Bbb{Z}$.
For each $j\in \Bbb{Z}$,  there is a
decomposition of $k$-vector spaces $M^j= B^j \oplus H^j \oplus
C^j$ with $B^j = \mathrm {Im} \,
\partial^{j-1}$ and $\mathrm {Ker}\,
\partial^j = B^j \oplus H^j$. Define $\tau^{\ge
i}(M)$ as the DG $\mathcal{A}$-submodule of $M$ such that
\begin{align*}
[\tau^{\ge
i}(M)]^j=\begin{cases}
        0, \quad \text{if}\,\, j\le i-1,\\
        C^i\oplus H^i, \,\,\text{if}\,\, j=i,\\
        M^j, \,\,\text{if}\,\, j>i.
\end{cases}
\end{align*}
The inclusion map $\iota: \tau^{\ge
i}(M)\to M$ is a quasi-isomorphism of DG $\mathcal{A}$-modules.
Let $o(\tau^{\ge
i}(M))$ be the DG $\mathcal{A}$-submodule of $\tau^{\ge i}(M)$ such that
\begin{equation*}
o(\tau^{\ge i}(M))^m =
\begin{cases} C^i,\quad\text{if}\,\, m=i,\\
 [\tau^{\ge
i}(M)]^m\quad\text{if}\,\, m\neq i.
\end{cases}
\end{equation*}
Obviously, $o(\tau^{\ge i}(M))$ is quasi-trivial. Hence the
canonical morphism $$\tau^{\ge i}(M) \to \tau^{\ge i}(M)/o(\tau^{\ge
i}(M))= \Sigma^i {}_{\mathcal{A}}k $$ is a quasi-isomorphism. Therefore,
$M\cong\Sigma^i {}_{\mathcal{A}}k$ in $\mathscr{D}(\mathcal{A})$.
\end{proof}
\begin{lem}\label{lrgor}
Let $M$ be a DG $\mathcal{A}^e$-module such that ${}_{\mathcal{A}}M\cong \Sigma^i(\mathcal{A}_{\mathcal{A}})'$ and $M_{\mathcal{A}}\cong \Sigma^i({}_{\mathcal{A}}\mathcal{A})'$.
Then there is an invertible DG $\mathcal{A}$-bimodule $L=\Sigma^i\mathcal{A}(\psi)$ for some $\psi\in \mathrm{Aut}_{dg}(\mathcal{A})$ such that  $M=L\otimes_{\mathcal{A}}\mathcal{A}'$.
\end{lem}
\begin{proof}
By the assumption, we have ${}_{\mathcal{A}}(\Sigma^iM')\cong {}_{\mathcal{A}}\mathcal{A}$ and $(\Sigma^iM')_{\mathcal{A}}\cong \mathcal{A}_{\mathcal{A}}$. Let $e$ be the generator of
${}_{\mathcal{A}}(\Sigma^iM')$. Then for any $a\in \mathcal{A}$, there exists unique $l_a\in \mathcal{A}$ such that
$e\cdot a=l_ae.$ This induces a map $\phi:\mathcal{A}\to \mathcal{A}$ such that $\phi(a)=l_a$. It is straightforward to check that $\phi$ is an isomorphism of the graded algebra $\mathcal{A}$.
Furthermore, $\phi$ is a chain map since
\begin{align*}
\partial_{\mathcal{A}}(\phi(a))e&=\partial_{\mathcal{A}}(l_a)e\\
&=\partial_{\Sigma^iM'}(l_ae)\\
&=\partial_{\Sigma^iM'}(e\cdot a)\\
&=e\cdot \partial_{\mathcal{A}}(a)\\
&=l_{\partial_{\mathcal{A}}(a)}e =\phi(\partial_{\mathcal{A}}(a))e,
\end{align*}
for any $a\in \mathcal{A}$. So $\phi\in \mathrm{Aut}_{dg}\mathcal{A}$ and $\Sigma^iM'\cong \mathcal{A}(\phi)$ as a DG $\mathcal{A}^e$-module.
Therefore, the DG $\mathcal{A}$-bimodule \begin{align*}
M\cong \Sigma^i(\mathcal{A}(\phi))'&\cong \Sigma^i\Hom_k(\mathcal{A}\otimes_{\mathcal{A}}\mathcal{A}(\phi),k)\\
                                   &\cong \Sigma^i\Hom_{\mathcal{A}}(\mathcal{A}(\phi),\Hom_k(\mathcal{A},k))\\
                                   &\cong \Sigma^i\Hom_{\mathcal{A}}(\mathcal{A}(\phi),\mathcal{A}\otimes_{\mathcal{A}} \mathcal{A}')\\
                                   &\cong \Sigma^i\Hom_{\mathcal{A}}(\mathcal{A}(\phi), \mathcal{A})\otimes_{\mathcal{A}} \mathcal{A}'\\
                                   &\cong \Sigma^i\mathcal{A}(\phi^{-1})\otimes_{\mathcal{A}} \mathcal{A}'.
\end{align*}
Let $\psi=\phi^{-1}$ and $L=\Sigma^i\mathcal{A}(\psi)$. Then the DG $\mathcal{A}$-bimodule $L$ is invertible and $M\cong L\otimes_{\mathcal{A}}\mathcal{A}'$.
\end{proof}

\section {local cohomology functor in dg setting}
Assume that $M$ is a DG module over a connected cochain DG algebra $\mathcal{A}$. Its
 subset $$\Gamma_{\frak{m}}(M) = \{m\in M|\mathcal{A}^{\ge
n}m = 0 ,\,\, \text{for some}\,\, n \ge 1\}$$ is actually a DG
$\mathcal{A}$-submodule. Indeed, for any $m\in \Gamma_{\frak{m}}(M)$
and $a\in \mathcal{A}$, there exists an integer $n\ge 1$ such that $\mathcal{A}^{\ge n}m = 0$ and hence
 $\mathcal{A}^{\ge n}(am)\subseteq \mathcal{A}^{\ge n}m = 0$. Then $am\in
\Gamma_{\frak{m}}(M)$. Moreover, we have $$0=\partial_M(\mathcal{A}^{\ge
n}m)=\partial_{\mathcal{A}}(\mathcal{A}^{\ge n})m + \mathcal{A}^{\ge n}\partial_M(m)$$ and
$\partial_{\mathcal{A}}(\mathcal{A}^{\ge n})m\subseteq \mathcal{A}^{\ge n}m =0$.  So
$\partial_M(m)\in \Gamma_{\frak{m}}(M)$. We denote $\Gamma_{\frak{m}}(M)$ by the torsion submodule of $M$.
Let $f: M\to N$ be a morphism of DG $\mathcal{A}$-modules.
Then $f(\Gamma_{\frak{m}}(M)) \subseteq
\Gamma_{\frak{m}}(N)$ since $f$ is $\mathcal{A}$-linear. Hence there exists a
morphism of DG $\mathcal{A}$-modules $\Gamma_{\frak{m}}(f):\Gamma_{\frak{m}}M
\to \Gamma_{\frak{m}}N$ which agrees with $f$ on each element of
$\Gamma_{\mathrm{m}}(M)$. Furthermore,
 if $g: M\to N, h: N\to L$ are morphisms of DG
$\mathcal{A}$-modules and $a\in \mathcal{A}$, then \begin{align*}
\begin{cases}
\Gamma_{\frak{m}}(h\circ f) =
\Gamma_{\frak{m}}(h)\circ \Gamma_{\frak{m}}(f),\\
\Gamma_{\frak{m}}(f + g) = \Gamma_{\frak{m}}(f) +
\Gamma_{\frak{m}}(g),\\
\Gamma_{\frak{m}}(\mathrm{id}_M) =
\mathrm{id}_{\Gamma_{\frak{m}}(M)}.
\end{cases}
\end{align*}
Thus $\Gamma_{\frak{m}}(-)$ is
a covariant  functor from $\mathscr{C}(\mathcal{A})$ to itself. Since the torsion functor
$$\Gamma_{\frak{m}}(-) = \varinjlim_{n\ge 1}\Hom_{\mathcal{A}}(\mathcal{A}/\mathcal{A}^{\ge
n},-),$$ it is left exact. By the definition of torsion functor above, it is clear that the
functors $\Hom_{\mathcal{A}}(M,-)$ and $\Hom_{\mathcal{A}}(M,\Gamma_{\frak{m}}(-))$ are
naturally isomorphic, if $M$ is a finite length DG $\mathcal{A}$-module.
 We write $R\Gamma_{\frak{m}}(-)$ as
the right derived functor of
$\Gamma_{\frak{m}}(-)$.

For
any object $M$ in $\mathscr{C}(\mathcal{A})$, its local cohomology $H_{\frak{m}}(M)$ is defined to be the
graded $H(\mathcal{A})$-module $H(R\Gamma_{\frak{m}}(M))$. Similarly, we have functors $\Gamma_{\frak{m}^{op}}$ and $\Gamma_{\frak{m}_{\mathcal{A}^e}}$ on $\mathscr{C}(\mathcal{A}^{op})$ and $\mathscr{C}(\mathcal{A}^e)$ respectively.
We use the same notations for some variants of these functors.
For example, $\Gamma_{\frak{m}}$ defines a functor from $\mathscr{C}(\mathcal{A}\otimes \mathcal{B})$ to $\mathscr{C}(\mathcal{A}\otimes \mathcal{B})$, where $\mathcal{B}$ is another connected cochain DG algebra.

\begin{lem}\label{comlim}
Let $\mathcal{A}$ be a homologically smooth and locally finite connected cochain DG
algebra. Then $R\Gamma_{\frak{m}}(-)$ commutes with direct limits.
\end{lem}
\begin{proof}
Since $\mathcal{A}$ is homologically smooth, ${}_{\mathcal{A}}k$ is a compact DG $\mathcal{A}$-module. For
any integer $n\ge 1$,  $\mathcal{A}/\mathcal{A}^{\ge n}$ is finite dimensional. Hence
$\mathcal{A}/\mathcal{A}^{\ge n}$ is a compact DG $\mathcal{A}$-module by Remark \ref{HSM}. Let $F$ be a minimal
semi-free resolution of $\mathcal{A}/\mathcal{A}^{\ge n}$. Then $F$ admits a finite
semi-basis. Hence $R\Hom_{\mathcal{A}}(\mathcal{A}/\mathcal{A}^{\ge n},-) = \Hom_{\mathcal{A}}(F,-)$ commutes
with direct limits. This implies that $R\Gamma_{\frak{m}}(-) =
\varinjlim_{n\ge 1}R\Hom_{\mathcal{A}}(\mathcal{A}/\mathcal{A}^{\ge n},-)$ commutes with direct
limits.
\end{proof}

\begin{lem}\label{torfindim}
Assume that $\mathcal{A}$ is a homologically smooth locally finite  connected cochain DG
algebra.  If
 $M$ is a direct limit of
DG $\mathcal{A}$-modules with finite dimensional cohomology, then the
canonical map $$\iota: R\Gamma_{\frak{m}}(M)\to M$$ is a
quasi-isomorphism.
\end{lem}
\begin{proof}
By Lemma \ref{comlim}, it suffices to show that the canonical map
$R\Gamma_{\frak{m}}(M)\to M$ is a quasi-isomorphism, if $M$ is a DG
$\mathcal{A}$-module with finite dimensional cohomology. Since $\mathcal{A}$ is
homologically smooth, $k_{\mathcal{A}}$ is compact. By Remark \ref{HSM}, $M'$
is a compact DG $\mathcal{A}\!^{op}$-module. There is a minimal semi-free
resolution $\gamma: F_{M'}\to M'$ such that $F_{M'}$ has a finite
semi-basis. This implies $\Gamma_{\frak{m}}((F_{M'})') = (F_{M'})'$.
Since $H(M)$ is finite dimensional,  $\theta: M\to M''$
  defined by $\theta(m)(f)=(-1)^{|f|\cdot|m|}f(m)$ is a
  quasi-isomorphism.
So $\Hom_k(\gamma,k)\circ\theta: M\to (F_{M'})'$ is a semi-injective
resolution of $M$. By the following commutative diagram

\begin{align*}
\xymatrix{H(R\Gamma_{\frak{m}}(M))\ar[r]^{H(\iota)}\ar[d]^{\mathrm{=}}
& H(M) \ar[d]^{H(\Hom_k(\gamma,k)\circ\theta)}
\\
H(\Gamma_{\frak{m}}((F_{M'})')) \ar[r]^{=} &H((F_{M'})'),
\\}\end{align*}
we get that $H(\iota)$ is an isomorphism.  Hence $\iota$ is a
quasi-isomorphism.

\end{proof}

\begin{prop}
Assume that $E$ is a semi-injective DG $\mathcal{A}\otimes \mathcal{B}$-module. For any
DG $\mathcal{B}$-module $Q$, $\Hom_{\mathcal{B}}(Q,E)$ is a semi-injective DG $\mathcal{A}$-module,
and $\Gamma_{\frak{m}_{\mathcal{B}}}(E)$ is a direct limit of semi-injective DG
$\mathcal{A}$-modules.
\end{prop}
\begin{proof}
By Lemma \ref{tensorcong}, one sees that $\Hom_{\mathcal{A}}(-,\Hom_{\mathcal{B}}(Q,E))=
\Hom_{\mathcal{A}\otimes \mathcal{B}}(-\otimes Q,E)$. This implies that $\Hom_{\mathcal{B}}(Q,E)$ is a
semi-injective DG $\mathcal{A}$-module. Hence $\Gamma_{\frak{m}_{\mathcal{B}}}(E) =
\varinjlim_{n\ge 1}\Hom_{\mathcal{B}}(\mathcal{B}/\mathcal{B}^{\ge n},E)$ is a direct limit of
semi-injective DG $\mathcal{A}$-modules.
\end{proof}

\section {local duality theorem and its applications}
 The notion of Gorenstein DG algebras was first
introduced by F$\acute{\mathrm{e}}$lix-Halperin-Thomas (cf.
\cite{FHT1}). In algebraic topology, Gorenstein DG algebras are
algebraic model of Gorenstein spaces. In the last twenty years,
Gorenstein DG algebras have been intensively studied in literature (cf.\cite{FJ,FHT1,Dual,Gam,MW2}). This section gives a
characterization of Gorensteinness for
homologically smooth connected cochain DG algebras by using local
cohomology functor.

\begin{thm}[Local Duality Theorem]\label{localdual}
Let $M$ be an object in $\mathscr{D}^c(\mathcal{A})$. Then
$$R\Gamma_{\frak{m}}(M)' = R\Hom_{\mathcal{A}}(M,R\Gamma_{\frak{m}}(\mathcal{A})')$$ in
$\mathscr{D}(\mathcal{A}\!^{op})$.
\end{thm}
\begin{proof}
Since $M$ is compact,  it admits a minimal semi-free
resolution $F_M$ with
a finite semi-basis. Let $\eta: \mathcal{A} \stackrel{\simeq}{\to} I$ be a
semi-injective resolution of the DG $\mathcal{A}^e$-module $\mathcal{A}$. Since
$$\Hom_{\mathcal{A}}(-,I) \cong \Hom_{\mathcal{A}}(-,\Hom_{\mathcal{A}^e}(\mathcal{A}^e,I))\cong
\Hom_{\mathcal{A}^e}(\mathcal{A}\otimes_k-,I),$$ ${}_{\mathcal{A}}I$ is semi-injective.  This implies
that $I\otimes_{\mathcal{A}}F_M$ is a semi-injective DG $\mathcal{A}$-module since $F_M$
is a semi-free DG $\mathcal{A}$-module with a finite semi-basis. Since
\begin{align*}
R\Hom_{\mathcal{A}}(M,R\Gamma_{\frak{m}}(\mathcal{A})')=\Hom_{\mathcal{A}}(F_M,\Gamma_{\frak{m}}(I)')\cong (\Gamma_{\frak{m}}(I)\otimes_{\mathcal{A}} F_M)'
\end{align*}
and $$R\Gamma_{\frak{m}}(M)'=
R\Gamma_{\frak{m}}(F_M)'=(\Gamma_{\frak{m}}(I\otimes_{\mathcal{A}}F_M))'\cong
(\Gamma_{\frak{m}}(I)\otimes_{\mathcal{A}}F_M)',$$ we have
$$R\Gamma_{\frak{m}}(M)' = R\Hom_{\mathcal{A}}(M,R\Gamma_{\frak{m}}(\mathcal{A})')$$ in
$\mathscr{D}(\mathcal{A}\!^{op})$
\end{proof}

\begin{rem}\label{op}
For any compact DG $\mathcal{A}\!^{op}$-module $N$, we can similarly prove
that $$R\Gamma_{\frak{m}\!^{op}}(N)' =
R\Hom_{\mathcal{A}\!^{op}}(N,R\Gamma_{\frak{m}\!^{op}}(\mathcal{A})')$$ in
$\mathscr{D}(\mathcal{A})$.
\end{rem}

\begin{lem}\label{below}
Assume that $\mathcal{A}$ is a left Gorenstein locally finite connected cochain DG algebra. Then
$R\Gamma_{\frak{m}}(\mathcal{A})$ is a direct limit of DG $\mathcal{A}^e$-modules with
finite dimensional cohomology, and $H_{\frak{m}}(\mathcal{A})'$ is bounded
below.
\end{lem}
\begin{proof}
Let $I$ be a semi-injective resolution of $\mathcal{A}$ as a DG $\mathcal{A}^e$-module.
Since $\mathcal{A}$ is left Gorenstein, we have $$H(R\Hom_{\mathcal{A}}(k,\mathcal{A}))= H(\Hom_{\mathcal{A}}(k,I))= \Sigma^j k$$
for some $j\in \Bbb{Z}$. For any $n\ge 1$,  $H(\Hom_{\mathcal{A}}(\mathcal{A}/\mathcal{A}^{\ge
n},I))$ is finite dimensional since $\mathcal{A}/\mathcal{A}^{\ge n}$ is finite
dimensional. Hence $\Gamma_{\frak{m}}(I)= \varinjlim_{n\ge
1}\Hom_{\mathcal{A}}(\mathcal{A}/\mathcal{A}^{\ge n},I)$ is a direct limit of DG $\mathcal{A}^e$-modules
with finite dimensional cohomology.

Since $\mathcal{A}/\mathcal{A}^{\ge n}$ is concentrated in degrees from $0$ to
$n-1$ for each $n\ge 1$, $H(\Hom_{\mathcal{A}}(\mathcal{A}/\mathcal{A}^{\ge n},I))$ is
concentrated in degrees from  $-j+1-n$ to $-j$. Hence
$H(\Gamma_{\frak{m}}(I))'\cong H(\varinjlim_{n\ge
1}\Hom_{\mathcal{A}}(\mathcal{A}/\mathcal{A}^{\ge n},I))'\cong (\varinjlim_{n\ge 1}
H(\Hom_{\mathcal{A}}(\mathcal{A}/\mathcal{A}^{\ge n},I)))'$ is concentrated in degrees
$\ge j$.

\end{proof}

\begin{prop}\label{smoothGor}
Let $\mathcal{A}$ be a homologically smooth locally finite connected cochain DG algebra. Then
$\mathcal{A}$ is Gorenstein if and only if $R\Gamma_{\frak{m}}(\mathcal{A})\cong
\Sigma^i({}_{\mathcal{A}}\mathcal{A})'$ in $\mathscr{D}(\mathcal{A}\!^{op})$.
\end{prop}

\begin{proof}
By Lemma \ref{torfindim},  we have $H_{\frak{m}}(k) = k$. If $\mathcal{A}$ is
Gorenstein, then there exists $i\in \Bbb{Z}$ such that
$R\Hom_{\mathcal{A}}(k,\mathcal{A})\cong \Sigma^ik_{\mathcal{A}}$ in $\mathscr{D}(\mathcal{A}\!^{op})$. By
Theorem \ref{localdual},
\begin{align*}
H_{\frak{m}}(k)' & \cong H(R\Hom_{\mathcal{A}}(k,R\Gamma_{\frak{m}}(\mathcal{A})') \\
&\cong H(R\Hom_{\mathcal{A}}(k,\mathcal{A}){}^L\otimes_{\mathcal{A}} R\Gamma_{\frak{m}}(\mathcal{A})') \\
&\cong H(\Sigma^ik{}^L\otimes_{\mathcal{A}}R\Gamma_{\frak{m}}(\mathcal{A})').
\end{align*}
Hence $H(\Sigma^ik{}^L\otimes_{\mathcal{A}}R\Gamma_{\frak{m}}(\mathcal{A})') \cong k$.
 By Lemma \ref{below} and \cite[Proposition 2.4]{MW1}, there is a minimal
 semi-free resolutions ${}_{\mathcal{A}}F \stackrel{\simeq}{\to}
 {}_{\mathcal{A}}(R\Gamma_{\frak{m}}(\mathcal{A})')$. Since
$$H(\Sigma^ik{}^L\otimes_{\mathcal{A}}R\Gamma_{\frak{m}}(\mathcal{A})') \cong
\Sigma^ik\otimes_{\mathcal{A}} F =k, $$ we have  $F = \Sigma^{-i}{}_{\mathcal{A}}\mathcal{A}$.
This implies that $R\Gamma_{\frak{m}}(\mathcal{A})_{\mathcal{A}}\cong\Sigma^i ({}_{\mathcal{A}}\mathcal{A})'$
in $\mathscr{D}(\mathcal{A}\!^{op})$.

Conversely, if $R\Gamma_{\frak{m}}(\mathcal{A})_{\mathcal{A}}\cong \Sigma^i({}_{\mathcal{A}}\mathcal{A})'$ in
$\mathscr{D}(\mathcal{A}\!^{op})$, then by Theorem \ref{localdual}, we have
$$H(R\Hom_{\mathcal{A}}(k,\mathcal{A}))\cong H(R\Hom_{\mathcal{A}}(k, \Sigma^i
R\Gamma_{\frak{m}}(\mathcal{A})'))\cong \Sigma^iH_{\frak{m}}(k) =
\Sigma^ik.$$ Since $\mathcal{A}$ is homologically smooth, $\mathcal{A}$ is Gorenstein
by Remark \ref{HSM}.
\end{proof}

\begin{rem}
For any homologically smooth locally finite connected cochain DG algebra $\mathcal{A}$, we can
similarly prove that $\mathcal{A}$ is Gorenstein if and only if
$R\Gamma_{\frak{m}\!^{op}}(\mathcal{A})\cong \Sigma^i(\mathcal{A}_\mathcal{A})'$ in
$\mathscr{D}(\mathcal{A})$, by Remark \ref{op} and Remark \ref{HSM}.
\end{rem}

\begin{thm}\label{smoothgoren}
Let $\mathcal{A}$ be a Noetherian, homologically smooth connected cochain DG
algebra. Then $\mathcal{A}$ is Gorenstein if and only if
$R\Gamma_{\frak{m}}(\mathcal{A})\cong  L \otimes_{\mathcal{A}}\mathcal{A}'$ in $\mathscr{D}(\mathcal{A}^e)$ for
some invertible DG $\mathcal{A}$-bimodule $L=\Sigma^i \mathcal{A}(\phi)$ with $\phi\in \mathrm{Aut}_{dg}(\mathcal{A})$.
\end{thm}
\begin{proof}
Assume that $R\Gamma_{\frak{m}}(\mathcal{A})\cong L \otimes_{\mathcal{A}}\mathcal{A}'$ in
$\mathscr{D}(\mathcal{A}^e)$, where $L=\Sigma^i \mathcal{A}(\phi)$ with $\phi\in \mathrm{Aut}_{dg}(\mathcal{A})$.
Then
\begin{align*}
R\Hom_{\mathcal{A}}(k, R\Gamma_{\frak{m}}(\mathcal{A})') &\cong R\Hom_{\mathcal{A}}(k, \Hom_{k}(L\otimes_{\mathcal{A}}\mathcal{A}',k))\\
                                                         &\stackrel{(1)}{\cong} R\Hom_{\mathcal{A}}(k, \Hom_{\mathcal{A}\!^{op}}(L,\mathcal{A}''))\\
                                                         &\stackrel{(2)}{\cong} R\Hom_{\mathcal{A}}(k, \Hom_{\mathcal{A}\!^{op}}(L,\mathcal{A}))\\
                                                         &\stackrel{(3)}{\cong} R\Hom_{\mathcal{A}}(k, \mathcal{A})\otimes_{\mathcal{A}} \Hom_{\mathcal{A}\!^{op}}(L,\mathcal{A}),
\end{align*}
where $(1)$, $(2)$ and $(3)$ are obtained by the fact that $k\in \mathscr{D}^c(\mathcal{A})$.
On the other hand, we have $R\Hom_{\mathcal{A}}(k,  R\Gamma_{\frak{m}}(\mathcal{A})')
\cong R\Gamma_{\frak{m}}(k)'$ by Theorem \ref{localdual}. By Lemma
\ref{torfindim} and Lemma \ref{onedim}, $(R\Gamma_{\frak{m}}(k)')_{\mathcal{A}}
\cong k_{\mathcal{A}}$ in $\mathscr{D}(\mathcal{A}\!^{op})$. Therefore,
$$R\Hom_{\mathcal{A}}(k,\mathcal{A})\otimes_{\mathcal{A}}\Hom_{\mathcal{A}\!^{op}}(L,\mathcal{A}) \cong k_{\mathcal{A}}$$ in
$\mathscr{D}(\mathcal{A}\!^{op})$. Since $L=\Sigma^i\mathcal{A}(\phi)$, we have $R\Hom_{\mathcal{A}}(k,\mathcal{A})\cong k_{\mathcal{A}}\otimes_{\mathcal{A}}
\Sigma^i\mathcal{A}(\phi)$. Thus $\mathcal{A}$ is Gorenstein.

Conversely, if $\mathcal{A}$ is Gorenstein, we should prove that
$R\Gamma_{\frak{m}}(\mathcal{A})\cong  L\otimes_{\mathcal{A}}\mathcal{A}'$ in $\mathscr{D}(\mathcal{A}^e)$ for
 some invertible DG $\mathcal{A}$-bimodule $L$. By Lemma \ref{lrgor}, it suffices to prove that $R\Gamma_{\frak{m}}(\mathcal{A})\cong
\Sigma^i({}_{\mathcal{A}}\mathcal{A})'$ in $\mathscr{D}(\mathcal{A}\!^{op})$ and
$R\Gamma_{\frak{m}}(\mathcal{A})\cong \Sigma^i(\mathcal{A}_{\mathcal{A}})'$ in $\mathscr{D}(\mathcal{A})$.

 By Proposition \ref{smoothGor},  $R\Gamma_{\frak{m}}(\mathcal{A})\cong
\Sigma^i({}_{\mathcal{A}}\mathcal{A})'$ in $\mathscr{D}(\mathcal{A}\!^{op})$, for some $i\in \Bbb{Z}$.
 Let
$\mathcal{A}\stackrel{\simeq}{\to} I$ be a semi-injective resolution of
${}_{\mathcal{A}}\mathcal{A}$. As a graded $\mathcal{A}^{\#}$-module, $\Hom_{\mathcal{A}}(k,I)^{\#} \cong
\bigoplus_{j\in J}\Sigma^{u_j} ({}_{\mathcal{A}^{\#}}k)$. Since $\mathcal{A}$ is
Gorenstein,  $\dim_kH(\Hom_{\mathcal{A}}(k,I))=1$. This implies that there is a
graded cocycle element $e\in \Hom_{\mathcal{A}}(k,I)$ of degree $i$ such
that $\Hom_{\mathcal{A}}(k,I)\cong {}_{\mathcal{A}}ke\oplus {}_{\mathcal{A}}X$ as a DG $\mathcal{A}$-module, where
$X$ is a quasi-trivial DG $\mathcal{A}$-module with $\frak{m}X=0$. Therefore,
$(\mathcal{A}_{\mathcal{A}})^{\#'}\otimes ke\oplus ((\mathcal{A}_{\mathcal{A}})^{\#'}\otimes X^{\#})$ is an
injective envelope of $\Hom_{\mathcal{A}}(k,I)^{\#}$. Now let $E^{\#}$ be an
injective $\mathcal{A}^{\#}$-module such that $\Hom_{\mathcal{A}}(k,I)^{\#}$ is an
essential graded $\mathcal{A}^{\#}$-submodule of $E^{\#}$. Then $E^{\#}\cong
(\mathcal{A}_{\mathcal{A}})^{\#'}\otimes ke\oplus ((\mathcal{A}_{\mathcal{A}})^{\#'}\otimes X^{\#})$ and
$E^{\#}\subseteq \Gamma_{\frak{m}}(I)^{\#}$. The graded
$A^{\#}$-module $I^{\#}$ can be decomposed as $I^{\#}\cong
E^{\#}\oplus J^{\#}$, for some injective $\mathcal{A}^{\#}$-module $J^{\#}$.
Since $\Hom_{\mathcal{A}}(k,I)^{\#} \bigcap J^{\#}=0$, it follows that
$\Gamma_{\frak{m}^{\#}}(J^{\#}) = 0$. Hence
$\Gamma_{\frak{m}}(I)^{\#}=\Gamma_{\frak{m}^{\#}}(I^{\#})\subseteq
E^{\#}.$ Therefore, $$ \Gamma_{\frak{m}}(I)^{\#} = \mathcal{A}^{\#'}\otimes
ke\oplus (\mathcal{A}^{\#'}\otimes X^{\#}).$$ We have
$$\Gamma_{\frak{m}}(I)^{\#'}\cong (ke)'\otimes \mathcal{A}_{\mathcal{A}^{\#}}\oplus
((\mathcal{A}_{\mathcal{A}})^{\#'}\otimes X^{\#})'.$$ Since $$\Gamma_{\frak{m}}(I)=
\varinjlim_{n\ge 1}\Hom_{\mathcal{A}}(\mathcal{A}/\mathcal{A}^{\ge n},I),$$ $\Hom_{\mathcal{A}}(k,I)$ is a DG
$\mathcal{A}$-submodule of $\Gamma_{\frak{m}}(I)$. Hence
$\partial_{\Gamma_{\frak{m}}(I)}(e) =0$. So
$\partial_{\Gamma_{\frak{m}}(I)'}|_{(ke)'} =0$. For any $a\in \mathcal{A}$, we
have $\partial_{\Gamma_{\frak{m}}(I)'}(e'\otimes a) = (-1)^{|e'|}
e'\otimes\partial_{\mathcal{A}}(a)$ since $\Gamma_{\frak{m}}(I)'$ is a DG
$\mathcal{A}\!^{op}$-module. This implies that $(ke)'\otimes \mathcal{A}_{\mathcal{A}}$ is a DG
$\mathcal{A}\!^{op}$-submodule of $\Gamma_{\frak{m}}(I)'$.
 Let $$\iota: (ke)'\otimes \mathcal{A}_{\mathcal{A}} \to \Gamma_{\frak{m}}(I)'$$
be the inclusion map. Since  $R\Gamma_{\frak{m}}(\mathcal{A})\cong
\Sigma^i({}_{\mathcal{A}}\mathcal{A})'$ in $\mathscr{D}(\mathcal{A}\!^{op})$, $H_{\frak{m}}(\mathcal{A})\cong
\Sigma^i H(\mathcal{A})'$ as a graded $k$-vector space. It is easy to check
that $H(\iota)$ is an isomorphism as a morphism of graded $k$-vector
space. This implies that $\iota$ is a quasi-isomorphism as a
morphism of DG $\mathcal{A}\!^{op}$-modules. Therefore, the composition of DG
morphisms $$\Gamma_{\frak{m}}(I)\to
\Gamma_{\frak{m}}(I)''\stackrel{\iota'}{\to} (\mathcal{A}_{\mathcal{A}})'\otimes (ke)'' $$
is also a quasi-isomorphism. Clearly, the DG $\mathcal{A}$-module
$$(\mathcal{A}_{\mathcal{A}})'\otimes (ke)''\cong (\mathcal{A}_{\mathcal{A}})'\otimes ke.$$ So
$\Gamma_{\frak{m}}(I) \simeq \Sigma^i (\mathcal{A}_{\mathcal{A}})'$ as a DG $\mathcal{A}$-module,
for some $i\in \Bbb{Z}$. By Lemma \ref{lrgor}, $R\Gamma_{\frak{m}}(\mathcal{A})\cong L \otimes_{\mathcal{A}} \mathcal{A}' $ in
$\mathscr{D}(\mathcal{A}^e)$, for some invertible DG $\mathcal{A}$-bimodule $L=\Sigma^i \mathcal{A}(\phi)$ with $\phi\in \mathrm{Aut}_{dg}(\mathcal{A})$.

\end{proof}

\begin{rem}\label{opcase}
By considering $\dim_kH(k{}^L\otimes_{\mathcal{A}}k)$,
one sees that $\mathcal{A}$ is homologically smooth if and only if $\mathcal{A}\!^{op}$ is homologically smooth. It follows that
$\mathcal{A}^{op}$ is also a homoloigcally smooth and Gorenstein DG algebra when the conditions in Theorem \ref{smoothGor} hold.
By Theorem \ref{smoothGor}, we have
$R\Gamma_{\frak{m}\!^{op}}(\mathcal{A}\!^{op})\cong Q\otimes_{\mathcal{A}\!^{op}}(\mathcal{A}\!^{op})'\cong \mathcal{A}'\otimes_{\mathcal{A}}Q$
in $\mathscr{D}(\mathcal{A}^e)$, where $Q=\Sigma^i\mathcal{A}(\phi)$ is an invertible DG $\mathcal{A}$-bimodule with $\phi\in \mathrm{Aut}_{dg}(\mathcal{A})$.

\end{rem}

\begin{prop}\label{lrcong}
Let $\mathcal{A}$ be a Noetherian, homologically smooth and Gorenstein
connected cochain DG algebra. Then $$R\Gamma_{\frak{m}}(\mathcal{A})\cong
R\Gamma_{\frak{m}_{\mathcal{A}^e}}(\mathcal{A})\cong R\Gamma_{\frak{m}^{op}}(\mathcal{A}\!^{op})$$ in
$\mathscr{D}(\mathcal{A}^e)$.
\end{prop}

\begin{proof}
Let $\mathcal{A}\stackrel{\simeq}{\to} I$ be a semi-injective resolution of
the DG $\mathcal{A}^e$-module $\mathcal{A}$. By Lemma \ref{semi-inj}, $I$ is
semi-injective both as a DG $\mathcal{A}$-module and as a DG $\mathcal{A}^{op}$-module.
Since $\mathcal{A}$ is a Noetherian homologically smooth Gorenstein
connected cochain DG algebra, we have $R\Gamma_{\frak{m}}(\mathcal{A})\cong
 L\otimes_{\mathcal{A}}\mathcal{A}'$ in $\mathscr{D}(\mathcal{A}^e)$ by Theorem \ref{smoothgoren},
for some invertible DG $\mathcal{A}$-bimodule $L$. For any $n\ge 1$, let $F_{\mathcal{A}\!^{op}/(\mathcal{A}\!^{op})^{\ge n}}$ be the minimal semi-free resolution of the DG right $\mathcal{A}$-module $\mathcal{A}\!^{op}/(\mathcal{A}\!^{op})^{\ge n}$. Then
\begin{align*}
R\Gamma_{\frak{m}_{\mathcal{A}^e}}(\mathcal{A})&\cong \Gamma_{\frak{m}_{\mathcal{A}^e}}(I) \\
&= \varinjlim_{n\ge 1}\Hom_{\mathcal{A}^e}((\mathcal{A}^e)/(\mathcal{A}^e)^{\ge n},I) \\
&= \varinjlim_{n\ge 1}\Hom_{\mathcal{A}^e}((\mathcal{A}^e)/(\mathcal{A}^{\ge
n}\otimes \mathcal{A}\!^{op} + \mathcal{A}\otimes (\mathcal{A}\!^{op})^{\ge n}),I) \\
&=\varinjlim_{n\ge 1}\Hom_{\mathcal{A}^e}((\mathcal{A}/\mathcal{A}^{\ge n})\otimes
(\mathcal{A}\!^{op}/(\mathcal{A}\!^{op})^{\ge n}),I)\\
&=\varinjlim_{n\ge 1}\Hom_{\mathcal{A}\!^{op}}(\mathcal{A}\!^{op}/(\mathcal{A}\!^{op})^{\ge n},
\Hom_{\mathcal{A}}(\mathcal{A}/\mathcal{A}^{\ge
n},I)) \\
&=\varinjlim_{n\ge 1}\varinjlim_{m\ge
1}\Hom_{\mathcal{A}\!^{op}}(\mathcal{A}\!^{op}/(\mathcal{A}\!^{op})^{\ge n},
\Hom_{\mathcal{A}}(\mathcal{A}/\mathcal{A}^{\ge m},I))\\
&\cong \varinjlim_{n\ge 1}\varinjlim_{m\ge
1}\Hom_{\mathcal{A}\!^{op}}(F_{\mathcal{A}\!^{op}/(\mathcal{A}\!^{op})^{\ge n}},
\Hom_{\mathcal{A}}(\mathcal{A}/\mathcal{A}^{\ge m},I))\\
&\cong \varinjlim_{n\ge
1}\Hom_{\mathcal{A}\!^{op}}(F_{\mathcal{A}\!^{op}/(\mathcal{A}\!^{op})^{\ge n}},\varinjlim_{m\ge
1}
\Hom_{\mathcal{A}}(\mathcal{A}/\mathcal{A}^{\ge m},I))\\
&\cong \varinjlim_{n\ge
1}\Hom_{\mathcal{A}\!^{op}}(F_{\mathcal{A}\!^{op}/(\mathcal{A}\!^{op})^{\ge n}},L\otimes_{\mathcal{A}}\mathcal{A}')\\
&\cong \varinjlim_{n\ge
1}\Hom_{\mathcal{A}\!^{op}}(\mathcal{A}\!^{op}/(\mathcal{A}\!^{op})^{\ge n}, L\otimes_{\mathcal{A}}\mathcal{A}')\\
&= L\otimes_{\mathcal{A}}\mathcal{A}' \cong R\Gamma_{\frak{m}}(\mathcal{A})
\\
\end{align*}
in $\mathscr{D}(\mathcal{A}^e)$. Similarly, we can prove that
$R\Gamma_{\frak{m}^e}(\mathcal{A})\cong R\Gamma_{\frak{m}^{op}}(\mathcal{A}\!^{op})$ in
$\mathscr{D}(\mathcal{A}^e)$.

\end{proof}

\begin{thm}\label{biduality}
Let $\mathcal{A}$ be a Noetherian, homologically smooth and Gorenstein
connected cochain DG algebra. Then for any $M\in \mathscr{D}^c(\mathcal{A})$,
$R\Gamma_{\frak{m}}(M)'$ is a compact DG $\mathcal{A}\!^{op}$-module. And
there is a quasi-inverse contravariant equivalence of categories,
\begin{align*}
\xymatrix{&\mathscr{D}^{c}(\mathcal{A})\quad\quad\ar@<1ex>[r]^{R\Gamma_{\frak{m}}(-)'}&\quad\quad
\mathscr{D}^{c}(\mathcal{A}\!^{op})\ar@<1ex>[l]^{R\Gamma_{\frak{m}\!^{op}}(-)'}}.
\end{align*}
\end{thm}

\begin{proof}
Let $F_M$ be a minimal semi-free resolution of $M$. Since $M$ is
compact, $F_M$ has a finite semi-basis. By Theorem \ref{localdual}
and Theorem \ref{smoothgoren}, we have
\begin{align*}
R\Gamma_{\frak{m}}(M)' &\cong
R\Hom_{\mathcal{A}}(M,R\Gamma_{\frak{m}}(\mathcal{A})') \\
&\cong \Hom_{\mathcal{A}}(F_M,(L\otimes_{\mathcal{A}}\mathcal{A}')')\\
&\cong
\Hom_{\mathcal{A}}(F_M,\Hom_{\mathcal{A}\!^{op}}(L,\mathcal{A}''))\\
&\cong \Hom_{\mathcal{A}}(F_M,\Hom_{\mathcal{A}\!^{op}}(L,\mathcal{A}))\\
&\cong \Hom_{\mathcal{A}}(F_M,\mathcal{A})\otimes_{\mathcal{A}}\Hom_{\mathcal{A}\!^{op}}(L,\mathcal{A}) ,
\end{align*} where $L$
is an invertible DG bimodule over $\mathcal{A}$. By Lemma \ref{dginvert}, the DG $\mathcal{A}^e$-module
$L\cong \Sigma^i\mathcal{A}(\phi)$, for some $i\in \Bbb{Z}$ and $\phi\in \mathrm{Aut}_{dg}\mathcal{A}$.
This implies that $$\Hom_{\mathcal{A}}(F_M,\mathcal{A})\otimes_{\mathcal{A}}\Hom_{\mathcal{A}\!^{op}}(L,\mathcal{A})\in \mathscr{D}^c(\mathcal{A}\!^{op})$$
since $F_M$ admits a finite semi-basis. So
$R\Gamma_{\frak{m}}(M)'$ is a compact DG $\mathcal{A}\!^{op}$-module. To prove
the quasi-inverse contravariant equivalence of categories, it
suffices to show that
$R\Gamma_{\frak{m}\!^{op}}(R\Gamma_{\frak{m}}(M)')'\cong M$ in
$\mathscr{D}(\mathcal{A})$. By Remark \ref{op},
$$R\Gamma_{\frak{m}\!^{op}}(R\Gamma_{\frak{m}}(M)')'\cong
R\Hom_{\mathcal{A}\!^{op}}(R\Gamma_{\frak{m}}(M)',R\Gamma_{\frak{m}\!^{op}}(\mathcal{A}\!^{op})')$$
in $\mathscr{D}(\mathcal{A})$. By Remark \ref{opcase}, we have
$R\Gamma_{\frak{m}\!^{op}}(\mathcal{A}\!^{op})\cong \mathcal{A}'\otimes_{\mathcal{A}} Q $ in
$\mathscr{D}(\mathcal{A}^e)$, for some invertible DG $\mathcal{A}$-bimodule $Q$.
Proposition \ref{lrcong} indicates that $\mathcal{A}'\otimes_{\mathcal{A}} Q\cong L\otimes_{\mathcal{A}}\mathcal{A}'$  in
$\mathscr{D}(\mathcal{A}^e)$. Therefore,
\begin{align*}
R\Gamma_{\frak{m}\!^{op}}(R\Gamma_{\frak{m}}(M)')'&\cong
\Hom_{\mathcal{A}\!^{op}}( \Hom_{\mathcal{A}}(F_M,\mathcal{A})\otimes_{\mathcal{A}}\Hom_{\mathcal{A}\!^{op}}(L,\mathcal{A}) , (\mathcal{A}'\otimes_{\mathcal{A}} Q)')\\
&\cong \Hom_{\mathcal{A}\!^{op}}( \Hom_{\mathcal{A}}(F_M,\mathcal{A})\otimes_{\mathcal{A}}\Hom_{\mathcal{A}\!^{op}}(L,\mathcal{A}), (L\otimes_{\mathcal{A}} \mathcal{A}')')\\
&\cong \Hom_{\mathcal{A}\!^{op}}( \Hom_{\mathcal{A}}(F_M,\mathcal{A})\otimes_{\mathcal{A}}\Hom_{\mathcal{A}\!^{op}}(L,\mathcal{A}), \Hom_{\mathcal{A}\!^{op}}(L,\mathcal{A}''))\\
&\cong
\Hom_{\mathcal{A}\!^{op}}(\Hom_{\mathcal{A}}(F_M,\mathcal{A}),\Hom_{\mathcal{A}\!^{op}}(\Hom_{\mathcal{A}\!^{op}}(L,\mathcal{A}),\Hom_{\mathcal{A}\!^{op}}(L,\mathcal{A}))\\
&\cong  \Hom_{\mathcal{A}\!^{op}}(\Hom_{\mathcal{A}}(F_M,\mathcal{A}),\mathcal{A})\\
&\cong F_M.
\end{align*}
So
$R\Gamma_{\frak{m}^{op}}(R\Gamma_{\frak{m}}(M)')'\cong M$ in
$\mathscr{D}(\mathcal{A})$.
\end{proof}

\section {homological determinant}
In this section, we will introduce and study the notion of homological determinant for group actions on Gorenstein homologically smooth DG algebras. For these, we need the following lemma first.

\begin{lem}\label{implem}
Let $\mathcal{A}$ be a Gorenstein  homologically smooth locally finite connected cochain DG algebra. For any $\sigma\in \mathrm{Aut}_{dg}(\mathcal{A})$, consider the morphism of DG $\mathcal{A}\!^{op}$-modules $$\Sigma^{-i}\sigma^{-1}:\Sigma^{-i}\mathcal{A}_{\mathcal{A}}\to \Sigma^{-i}\mathcal{A}_{\mathcal{A}},$$ which is $\sigma^{-1}$-linear. The Matlis dual map $(\Sigma^{-i} \sigma^{-1})':\Sigma^{i}{}_{\mathcal{A}}\mathcal{A}'\to \Sigma^{i}{}_{\mathcal{A}}\mathcal{A}'$ is $\sigma$-linear. Hence there exists a scalar $c\in k^{\times}$ such that the $H(\sigma)$-linear map
$$H_{\frak{m}}(\sigma): H_{\frak{m}}(\mathcal{A})\to H_{\frak{m}}(\mathcal{A})$$ is equal to
$$cH[(\Sigma^{-i} \sigma^{-1})']:\Sigma^i{}_{H(\mathcal{A})}[H(\mathcal{A})]'\to \Sigma^i{}_{H(\mathcal{A})}[H(\mathcal{A})]'.$$
\end{lem}
\begin{proof}
By Theorem \ref{smoothgoren}, $H_{\frak{m}}(\mathcal{A})$ is equal to $\Sigma^i{}_{H(\mathcal{A})}[H(\mathcal{A})]'$, for some $i\in \Bbb{Z}$.
Acting $H_{\frak{m}}(-)$ on the $\mathcal{A}$-linear map $\sigma: {}_{\mathcal{A}}\mathcal{A}\to {}_{\mathcal{A}}^{\sigma}\mathcal{A}$, we get the $H(\mathcal{A})$-linear map
$$H_{\frak{m}}(\sigma):H_{\frak{m}}(\mathcal{A})\to H_{\frak{m}}({}^{\sigma}\mathcal{A})={}^{H(\sigma)}H_{\frak{m}}(\mathcal{A}).$$
By Theorem \ref{smoothgoren}, we may write this as
$$H_{\frak{m}}(\sigma):  \Sigma^i{}_{H(\mathcal{A})}[H(\mathcal{A})]'\to \Sigma^i{}_{H(\mathcal{A})}^{H(\sigma)}[H(\mathcal{A})]'.$$
By Matlis duality, one sees that the only $H(\mathcal{A})$-linear maps from $\Sigma^i{}_{H(\mathcal{A})}[H(\mathcal{A})]'$ to $\Sigma^i{}_{H(\mathcal{A})}^{H(\sigma)}[H(\mathcal{A})]'$ are of the form $cH[(\Sigma^{-i} \sigma^{-1})']$, where $c\in k$.

Finally, we want to see $c\neq 0$. This is easy, since $H_{\frak{m}}(\sigma )H_{\frak{m}}(\sigma^{-1})=\mathrm{id}_{H_{\frak{m}}(\mathcal{A})}.$
\end{proof}

Lemma \ref{implem} enables us to give the following definition.
\begin{defn}\label{DGhdet}{\rm
 Assume that $\mathcal{A}$ is a Gorenstein, homologically smooth, locally finite connected cochain DG algebra. For each $\sigma\in \mathrm{Aut}_{dg}(\mathcal{A})$, we have $H_{\frak{m}}(\sigma)=cH[(\Sigma^{-i} \sigma^{-1})']$ by Lemma \ref{implem}, for some $c\in k^{\times}$ and $i\in \Bbb{Z}$. We write $\mathrm{Hdet}_{\mathcal{A}}\sigma =c^{-1}$, and call it the homological determinant of $\sigma$. It is easy for one to see that $\mathrm{Hdet}_{\mathcal{A}}$ defines a group homomorphism $\mathrm{Aut}_{dg}(\mathcal{A})\to k^{\times}$. }
\end{defn}

\begin{prop}\label{same}
Let $\mathcal{A}$ be a homologically smooth locally finite connected cochain DG algebra such that $H(\mathcal{A})$ is a Noetherian AS-Gorenstein algebra.
Then for any $\sigma\in \mathrm{Aut}_{dg} \mathcal{A}$, we have $\mathrm{Hdet}_{\mathcal{A}}\sigma=\mathrm{hdet}_{H(\mathcal{A})} H(\sigma )$.
\end{prop}
\begin{proof}
Since $H(\mathcal{A})$ is a Noetherian AS-Gorenstein algebra, we have
\begin{align*}
H^i_{\frak{m}_{H(\mathcal{A})}}(H(\mathcal{A}))\cong \begin{cases}
0, \quad\quad\quad\text{for}\quad i\neq d,\\
{}_{H(\mathcal{A})}H(\mathcal{A})'(l), \quad\text{for}\quad i=d,
\end{cases}
\end{align*}
by \cite[Lemma 2.1]{JoZ}.
Here $d$ and $l$ are the injective dimension and the Gorenstein index of $H(\mathcal{A})$, respectively.
So
\begin{align*}
\mathrm{Ext}_{H(\mathcal{A})}^i(k,H(\mathcal{A}))\cong \begin{cases}
0, \quad\quad\text{for}\quad i\neq d,\\
k(l), \quad\text{for}\quad i=d,
\end{cases}
\end{align*}
By \cite[Proposition 1]{Gam}, $\mathcal{A}$ is a Gorenstein DG algebra. By Theorme \ref{smoothgoren},
$R\Gamma_{\frak{m}}(\mathcal{A})\cong
\Sigma^j({}_{\mathcal{A}}\mathcal{A})'$ in $\mathscr{D}(\mathcal{A}\!^{op})$ for some $j\in \Bbb{Z}$. Then by Theorem \ref{localdual},
\begin{align*}
H(R\Hom_{\mathcal{A}}(k,\mathcal{A}))&\cong H(R\Hom_{\mathcal{A}}(k, \Sigma^j
R\Gamma_{\frak{m}}(\mathcal{A})'))\\
&\cong \Sigma^jH_{\frak{m}}(k) =
\Sigma^jk.
\end{align*}
By \cite[Theorem 4.7]{KM}, there is a natural sepctral sequence
$$E_2^{p,q}=\mathrm{Ext}_{H(\mathcal{A})}^{p,q}(k,H(\mathcal{A})) \Rightarrow H^{p+q}(R\Hom_{\mathcal{A}}(k,\mathcal{A})). $$
Hence $j=d-l$. Let $c=\mathrm{Hdet}_{\mathcal{A}}\sigma$. Then $$H_{\frak{m}}(\sigma)=c^{-1}H[(\Sigma^{l-d} \sigma^{-1})']:\Sigma^{d-l}{}_{H(\mathcal{A})}[H(\mathcal{A})]'\to \Sigma^{d-l}{}_{H(\mathcal{A})}[H(\mathcal{A})]'$$ by Definition \ref{DGhdet}.
We have \begin{align}\label{dgcase}
H_{\frak{m}}(\sigma)=c^{-1}H[(\Sigma^{l-d} \sigma^{-1})']
=c^{-1}\Sigma^{d-l}[H(\sigma^{-1})]'
=c^{-1}\Sigma^{d-l}[(H(\sigma))^{-1}]'.
\end{align}
By \cite[2.2 and 2.3]{JoZ}, $H(\sigma):H(\mathcal{A})\to H(\mathcal{A})$ induces a $H(\sigma)$-linear map $$H_{\frak{m}_{H(\mathcal{A})}}^d(H(\sigma)):{}_{H(\mathcal{A})}H(\mathcal{A})'(l)\to {}_{H(\mathcal{A})}H(\mathcal{A})'(l).$$
By the definition of the homological determinant of $H(\sigma)$, we have
\begin{align}\label{grcase}H_{\frak{m}_{H(\mathcal{A})}}^d(H(\sigma))=[\mathrm{hdet}_{H(\mathcal{A})}H(\sigma)]^{-1}\{[(H(\sigma))^{-1}]'(l)\}.
\end{align}
By (\ref{dgcase}) and (\ref{grcase}), we have $\mathrm{Hdet}_{\mathcal{A}}\sigma=c=\mathrm{hdet}_{H(\mathcal{A})}H(\sigma)$.
\end{proof}
\begin{prop}\label{equal}
Let $\mathcal{A}$ be a connected cochain DG algebra. Suppose that $G$ is a finite subgroup of  $\mathrm{Aut}_{dg}\mathcal{A}$.
Then
\begin{enumerate}
\item
$\mathcal{A}^{G}=\{a\in \mathcal{A}|g(a)=a, \,\forall \, g\in G\}$ is a connected cochain DG subalgebra of $\mathcal{A}$;
\item $H(G):\{H(g)| g\in G\}$ is a finite subgroup of $\mathrm{Aut}_{gr}H(\mathcal{A})$;
\item The graded fixed subalgebra $H(\mathcal{A})^{H(G)}$ of $H(\mathcal{A})$ is isomorphic to $H(\mathcal{A}^{G})$.
\end {enumerate}
\end{prop}
\begin{proof}
(1) It is straightford to show that $\mathcal{A}^{G}$ is a subcomplex which is closed under the addition and multiplication of $\mathcal{A}$.

(2) For any $g_1,g_2\in G$, we have $H(g_1)H(g_2)=H(g_1\circ g_2)\in H(G)$. Hence $H(G)$ is a subgroup of $\mathrm{Aut}_{gr}H(\mathcal{A})$.

(3) Let $\iota: \mathcal{A}^G\to \mathcal{A}$ be the inclusion morphism of DG algebras. It induces a morphism of graded algebras $H(\iota): H(\mathcal{A}^G)\to H(\mathcal{A})$. For any $\lceil z\rceil \in \mathrm{ker}(H(\iota))$, we have $z\in \mathcal{A}^G, \partial_{\mathcal{A}^G}(z)=0$ and $H(\iota)(\lceil z\rceil)=\lceil z\rceil=0$ in $H(A)$. There exists $a\in A$ such that $\partial_A(a)=z$. Let $x=\frac{1}{|G|^{-1}}\sum\limits_{g\in G} g(a)$. For any $\sigma\in G$, we have $$
\sigma(x)=\frac{1}{|G|^{-1}}\sum\limits_{g\in G} \sigma\circ g(a)=\frac{1}{|G|^{-1}}\sum\limits_{g'\in G} g'(a)=x$$
and \begin{align*}
\partial_{\mathcal{A}^G}(x)&=\partial_{\mathcal{A}^G}[\frac{1}{|G|^{-1}}\sum\limits_{g\in G} g(a)]\\
&=\frac{1}{|G|^{-1}}\sum\limits_{g\in G} \partial_A\circ g(a)\\
&=\frac{1}{|G|^{-1}}\sum\limits_{g\in G} g\circ \partial_A(a)\\
&=\frac{1}{|G|^{-1}}\sum\limits_{g\in G} g(z)=\frac{1}{|G|^{-1}}\sum\limits_{g\in G}z=z.
\end{align*}
Then $\lceil z\rceil =0$ in $H(\mathcal{A}^G)$ and $H(\iota)$ is a monomorphism.  Hence $H(\mathcal{A}^G)\cong \mathrm{im}(H(\iota))$.
For any $\lceil z \rceil\in H(\mathcal{A}^{G})$, we have $\partial_{\mathcal{A}^G}(z)=\partial_{\mathcal{A}}(z)=0$ and $g(z)=z$, for any $g\in G$.  Hence
$H(g)(\lceil z \rceil)= \lceil g(z) \rceil=\lceil z\rceil$, for any $g\in G$. This implies that $\lceil z \rceil\in H(\mathcal{A})^{H(G)}$. So $\mathrm{im}(H(\iota))\subseteq H(\mathcal{A})^{H(G)}$.
Conversely, let $\lceil b \rceil\in H(\mathcal{A})^{H(G)}$. Then $\lceil b\rceil=H(g)(\lceil b \rceil)=\lceil g(b)\rceil $ for any $g\in G$. So there exists $x_g\in \mathcal{A}$ such that $g(b)-b=\partial_{\mathcal{A}}(x_g)$. Set $$y=\frac{1}{|G|^{-1}}\sum\limits_{g\in G} g(b).$$ For any $\sigma\in G$, we have $$\sigma(y)=\sigma[\frac{1}{|G|^{-1}}\sum\limits_{g\in G} g(b)]=\frac{1}{|G|^{-1}}\sum\limits_{g\in G}\sigma\circ g(b)=\frac{1}{|G|^{-1}}\sum\limits_{g'\in G} g'(b)=y.$$
Furthermore,  \begin{align*}
\partial_{\mathcal{A}}(y)&=\partial_{\mathcal{A}}[\frac{1}{|G|^{-1}}\sum\limits_{g\in G} g(b)]\\
&=\frac{1}{|G|^{-1}}\sum\limits_{g\in G}\partial_{\mathcal{A}}\circ g(b)\\
&=\frac{1}{|G|^{-1}}\sum\limits_{g\in G} g(\partial_{\mathcal{A}}(b))=0
\end{align*}
and
\begin{align*}
\lceil \iota(y)\rceil=\lceil y\rceil &=\lceil \frac{1}{|G|^{-1}}\sum\limits_{g\in G} g(b) \rceil\\
&=\lceil  \frac{1}{|G|^{-1}}\sum\limits_{g\in G} (z+\partial_{\mathcal{A}}(x_g)) \rceil\\
&=\lceil b+ \partial_{\mathcal{A}}(\frac{1}{|G|^{-1}} \sum\limits_{g\in G} x_g) \rceil= \lceil b\rceil.
\end{align*}
Then $H(\mathcal{A})^{H(G)}\subseteq  \mathrm{im}(H(\iota))$ and hence $H(\mathcal{A})^{H(G)}= \mathrm{im}(H(\iota))\cong H(\mathcal{A}^{G})$.
\end{proof}

\begin{thm}\label{fixedsub}
Let $\mathcal{A}$ be a homologically smooth locally finite connected cochain DG algebra such that $H(\mathcal{A})$ is a Noetherian AS-Gorenstein graded algebra.
If $G$ is a finite subgroup of $\mathrm{Aut}_{dg}(\mathcal{A})$ such that the homological determinant of each $\sigma \in G$ is $1$,  then the fixed DG subalgebra $\mathcal{A}^G$ is a Gorenstein DG algebra.
\end{thm}
\begin{proof}
By Proposition \ref{same}, $\mathrm{hdet}_{H(\mathcal{A})}H(\sigma)=\mathrm{Hdet}_{\mathcal{A}}(\sigma)=1$ for any $\sigma\in G$. Hence $H(\mathcal{A})^{H(G)}$ is AS-Gorenstein by \cite[Theorem 3.3]{JoZ}. Since $H(\mathcal{A}^G)\cong H(\mathcal{A})^{H(G)}$ by
Proposition \ref{equal}, we conclude that $H(\mathcal{A}^G)$ is AS-Gorenstein. This implies that $\mathcal{A}^G$ is a Gorenstein DG algebra by
\cite[Proposition 1]{Gam}.
\end{proof}

DG down-up algebras are introduced and systematically studied in \cite{MHLX}. Recall that a connected cochain DG algebra $\mathcal{A}$ is called down-up if  $\mathcal{A}^{\#}$ is a graded down-up algebra $A(\alpha,\beta)$, which is a graded associative algebra generated by degree- one elements $x$ and $y$ over $k$ with the relations given by $$x^2y=\alpha xyx+\beta yx^2\quad\text{ and}\quad
 xy^2=\alpha yxy+\beta y^2x.$$ When $\beta\neq 0$, $A(\alpha,\beta)$ is Noetherian.
 By \cite[Theorem A]{MHLX}, the differential of a Noetherian DG down-up algebra can be described as follows:

 (a) $\partial_{\mathcal{A}}=0$ if $1+\alpha-\beta \neq 0$ or $1+\alpha-\beta =0, \beta^3\neq 1$;

(b) $\partial_{\mathcal{A}}$ is defined  by either
$$\begin{cases}
\partial_{\mathcal{A}}(x)=cy^2 \\
\partial_{\mathcal{A}}(y)=dx^2 \\
 c\cdot d=0,\,\, c, d\in k
\end{cases} \,\,
\text{or} \,\,\,\,
\begin{cases}
\partial_{\mathcal{A}}(x)= 2dx^2+cxy+cyx-\frac{c^2}{d}y^2\\
\partial_{\mathcal{A}}(y)= -\frac{d^2}{c}x^2 +dxy+dyx+ 2cy^2\\
c\in k^{\times}, d\in k^{\times}
\end{cases} $$ if  $1+\alpha-\beta =0, \beta^3=1$ and $\beta\neq 1$;

(c) $\partial_{\mathcal{A}}$ is defined by
$
\begin{cases}
\partial_{\mathcal{A}}(x)=c_1x^2 + c_2(xy+yx)+c_3y^2\\
\partial_{\mathcal{A}}(y)=d_1x^2+d_2(xy+yx)+d_3y^2\\
c_i, d_i \in k, i=1,2,3
\end{cases}$
if $\alpha =0$ and $\beta= 1$.

Therefore,  we can divide Noetherian
DG down-up algebras with non-trivial differentials into the
following three cases: $$
 \text{Case}\,1.\,\,  \mathcal{A}^{\#}=A(\xi-1,\xi), \quad   \text{Case}\,2. \,\, \mathcal{A}^{\#}=A(\xi^2-1,\xi^2), \quad
  \text{Case}\,3.\,\,  \mathcal{A}^{\#}=A(0,1), $$
 where $\xi$ is a fixed primitive cubic root of unity.  A Noetherian
 DG down-up algebra of the third case is not isomorphic to any Noetherian DG down-up algebras of the other two cases since we have $A(0,1)\not \cong A(\xi-1,\xi)$ and $A(0,1)\not \cong A(\xi^2-1,\xi^2)$ by \cite[Corollary 6.2]{BR}. By \cite[Theorem B]{MHLX},
the set of non-trivial Noetherian DG down-up algebras belonging to Case 1 or Case 2 can be classified into $3$ different isomorphism classes represented by
\begin{enumerate}
\item   $\mathcal{A}_1$ with $\mathcal{A}_1^{\#}=A(\xi-1,\xi)$ and $\begin{cases}
\partial_{\mathcal{A}_1}(x)=y^2 \\
\partial_{\mathcal{A}_1}(y)=0;
\end{cases}$
\item  $\mathcal{A}_2$ with $\mathcal{A}_2^{\#}=A(\xi-1,\xi)$ and $\begin{cases}
\partial_{\mathcal{A}_2}(x)=0 \\
\partial_{\mathcal{A}_2}(y)=x^2;
\end{cases}$
\item  $\mathcal{A}_3$ with $\mathcal{A}_3^{\#}=A(\xi-1,\xi)$ and $\begin{cases}
\partial_{\mathcal{A}_2}(x)=2x^2+xy+yx-y^2 \\
\partial_{\mathcal{A}_2}(y)=-x^2+xy+yx+2y^2.
\end{cases}$
\end{enumerate}
 By \cite[Proposition 5.5]{MHLX},  $$\begin{array}{lll}
 H(\mathcal{A}_1)&=&\frac{k\langle \lceil xy+yx\rceil, \lceil y\rceil\rangle}{\left(\begin{array}{ccc}
                                  \xi \lceil y\rceil \lceil xy+yx\rceil - \lceil xy+yx\rceil \lceil y\rceil \\
                                   \lceil y^2\rceil \\
                                    \end {array}\right)}, \\
 H(\mathcal{A}_2)&=& \frac{k\langle \lceil xy+yx\rceil, \lceil x\rceil\rangle}{\left(\begin{array}{ccc}
                                  \lceil x\rceil \lceil xy+yx\rceil - \xi \lceil xy+yx\rceil \lceil x\rceil \\
                                   \lceil x^2\rceil \\
                                    \end {array}\right)},\\
\text{and}\quad H(\mathcal{A}_3)&=&
k[\lceil(xy+yx)^3\rceil]
 \end{array}$$
are all AS-Gorenstein. Furthermore, $\mathcal{A}_1, \mathcal{A}_2$ and $\mathcal{A}_3$ are Calabi-Yau DG algebras by \cite[Proposition 6.1, Corollary 6.2 and Proposition 6.4]{MHLX}. For Case $3$,  $H(\mathcal{A})$ is AS-Gorenstein by Proposition $5.7$, and  $\mathcal{A}$ is Calabi-Yau by \cite[Theorem 6.11]{MHLX}.  Therefore, any DG down-up DG algebra $\mathcal{A}$ is homologically smooth and $H(\mathcal{A})$ is a Noetherian AS-Gorenstein graded algebra. Then we have the following corollary.
\begin{cor}\label{dgducase}
Let $\mathcal{A}$ be a Noetherian DG down-up algebra. If $G$ is a finite subgroup of $\mathrm{Aut}_{dg}(\mathcal{A})$ such that the homological determinant of each $\sigma \in G$ is $1$,  then the fixed DG subalgebra $\mathcal{A}^G$ is a Gorenstein DG algebra.
\end{cor}

In \cite{MXYA}, the notion of DG free algebra is introduced. Recall that a connected cochain DG algebra $\mathcal{A}$ is called DG free if $\mathcal{A}^{\#}$ is a free algebra $k\langle x_1,x_2,\cdots, x_n\rangle$ generated by degree $1$ elements. By \cite[Theorem 2.4]{MXYA}, $\partial_{A}$ is in one to one corresponce with a set of crisscross ordered n-tuples of $n \times n$
matrices. More precisely, there exists a crisscross ordered $n$-tuple $(M^1,M^2,\cdots, M^n)$ of $n\times n$ matrices such that $\partial_{\mathcal{A}}$ is defined by $$
\partial_{\mathcal{A}}(x_i)=(x_1,x_2,\cdots, x_n)M^i \left(
                         \begin{array}{c}
                          x_1 \\
                          x_2\\
                          \vdots\\
                          x_n\\
                         \end{array}
                       \right), \forall i\in \{1,2,\cdots, n\}.$$
 Here,  $(M^1,M^2,\cdots, M^n)$  with
$$M^i=(c^i_1,c^i_2,\cdots, c^i_n)=\left(
                         \begin{array}{c}
                        r^i_1 \\
                          r^i_2\\
                          \vdots\\
                          r^i_n\\
                         \end{array}
                       \right) $$
is called crisscross if
$$  \sum\limits_{l=1}^n[c_j^lr_l^i-c^i_lr_j^l]=(0)_{n\times n}, \forall i,j\in \{1,2,\cdots, n\}. $$ When $n=2$, we classify all isomorphism classes of DG free algebras. It is proved in \cite{MXYA} that all non-trivial DG free algebras generated by $2$ degree-one elements are Koszul Calabi-Yau DG algebras, which are obviously homologically smooth. Furthermore, their cohomology graded algebras were all Noetherian and AS-Gorenstein by \cite[Proposition 5.1]{MXYA}. Applying Theorem \ref{fixedsub}, we get the following corollary.

\begin{cor}\label{dgfreecase}
Let $\mathcal{A}$ be a non-trivial DG free algebra with $2$ degree-one generators. If $G$ is a finite subgroup of $\mathrm{Aut}_{dg}(\mathcal{A})$ such that the homological determinant of each $\sigma \in G$ is $1$,  then the fixed DG subalgebra $\mathcal{A}^G$ is a Gorenstein DG algebra.
\end{cor}

\section*{Acknowledgments}
X.-F. Mao is supported by the National Natural Science Foundation of
China (No. 11871326).

\def\refname{References}

\end{document}